\documentclass[11pt,a4paper]{article}

\usepackage{epsf,epsfig,amsfonts,amsgen,amsmath,amstext,amsbsy,amsopn,amsthm,cases,listings,color
}
\usepackage{ebezier,eepic}
\usepackage{color}
\usepackage{multirow}
\usepackage{epstopdf}
\usepackage{graphicx}
\usepackage{pgf,tikz}
\usepackage{mathrsfs}
\usepackage[marginal]{footmisc}
\usepackage{enumitem}
\usepackage[titletoc]{appendix}
\usepackage{booktabs}
\usepackage{url}
\usepackage{mathtools}

\usepackage{pgfplots}
\usepackage{authblk}
\usepackage{amssymb}

\usepackage{wasysym}

\usepackage{empheq}

\usepackage{dsfont}

\usepackage{subfigure}
\pgfplotsset{compat=1.18}
\usepackage{mathrsfs}
\usepackage{wasysym} 
\usetikzlibrary{arrows}
\usepackage{aligned-overset}
\usepackage{bm}
\usepackage[backref=page]{hyperref}
\allowdisplaybreaks[1]

\definecolor{uuuuuu}{rgb}{0.27,0.27,0.27}
\definecolor{sqsqsq}{rgb}{0.1255,0.1255,0.1255}

\newtheorem{definition}{Definition} [section]
\newtheorem{theorem}[definition]{Theorem}
\newtheorem{lemma}[definition]{Lemma}
\newtheorem{proposition}[definition]{Proposition}

\newtheorem{corollary}[definition]{Corollary}

\newtheorem{claim}[definition]{Claim}
\newtheorem{problem}[definition]{Problem}

\newtheorem{fact}[definition]{Fact}
\newenvironment{pf}{\noindent{\bf Proof}}

\begin{document}
\title{\bf\Large On the boundedness of degenerate hypergraphs}
\date{\today}
\author[1]{Jianfeng Hou\thanks{Research was supported by National Key R\&D Program of China (Grant No. 2023YFA1010202), National Natural Science Foundation of China (Grant No. 12071077), the Central Guidance on Local Science and Technology Development Fund of Fujian Province (Grant No. 2023L3003). Email: \texttt{jfhou@fzu.edu.cn}}}
\author[1]{Caiyun Hu\thanks{Research was supported by National Key R\&D Program of China (Grant No. 2023YFA1010202). Email: \texttt{hucaiyun.fzu@gmail.com}}}
\author[2]{Heng Li\thanks{Email: \texttt{heng.li@sdu.edu.cn}}}
\author[3]{Xizhi Liu\thanks{Research was supported by ERC Advanced Grant 101020255. Email: \texttt{xizhi.liu.ac@gmail.com}}}
\author[1]{Caihong Yang\thanks{Research was supported by National Key R\&D Program of China (Grant No. 2023YFA1010202). Email: \texttt{chyang.fzu@gmail.com}}}
\author[1]{Yixiao Zhang\thanks{Research was supported by National Key R\&D Program of China (Grant No. 2023YFA1010202). Email: \texttt{fzuzyx@gmail.com}}}
\affil[1]{Center for Discrete Mathematics,
            Fuzhou University, Fujian, 350003, China}
\affil[2]{School of Mathematics, Shandong University, 
            Shandong, 250100, China}
\affil[3]{Mathematics Institute and DIMAP,
            University of Warwick, 
            Coventry, CV4 7AL, UK}
\maketitle
\begin{abstract}
We investigate the impact of a high-degree vertex in Tur\'{a}n problems for degenerate hypergraphs (including graphs). 
We say an $r$-graph $F$ is bounded if there exist constants $\alpha, \beta>0$ such that for large $n$, every $n$-vertex $F$-free $r$-graph with a vertex of degree at least $\alpha \binom{n-1}{r-1}$ has fewer than $(1-\beta) \cdot \mathrm{ex}(n,F)$ edges. 
The boundedness property is crucial for recent works~\cite{HHLLYZ23a,DHLY24} that aim to extend the classical Hajnal--Szemer\'{e}di Theorem and the anti-Ramsey theorems of Erd\H{o}s--Simonovits--S\'{o}s. 

We show that many well-studied degenerate hypergraphs, such as all even cycles, most complete bipartite graphs, and the expansion of most complete bipartite graphs, are bounded. 
In addition, to prove the boundedness of the expansion of complete bipartite graphs, we introduce and solve a Zarankiewicz-type problem for $3$-graphs, strengthening a theorem by Kostochka--Mubayi--Verstra\"{e}te~\cite{KMV15}.

\medskip

\end{abstract}
\section{Introduction}\label{SEC:Intorduction}
Given an integer $r\ge 2$, an \textbf{$r$-uniform hypergraph} (henceforth \textbf{$r$-graph}) $\mathcal{H}$ is a collection of $r$-subsets of some finite set $V$.
We identify a hypergraph $\mathcal{H}$ with its edge set and use $V(\mathcal{H})$ to denote its vertex set. 
The size of $V(\mathcal{H})$ is denoted by $v(\mathcal{H})$. 
Given a vertex $v\in V(\mathcal{H})$, 
the \textbf{degree} $d_{\mathcal{H}}(v)$ of $v$ in $\mathcal{H}$ is the number of edges in $\mathcal{H}$ containing $v$.
We use $\delta(\mathcal{H})$, $\Delta(\mathcal{H})$, and $d(\mathcal{H})$ to denote the \textbf{minimum}, \textbf{maximum}, and \textbf{average degree} of $\mathcal{H}$, respectively.
We will omit the subscript $\mathcal{H}$ if it is clear from the context.

Given a family $\mathcal{F}$ of $r$-graphs, we say an $r$-graph $\mathcal{H}$ is \textbf{$\mathcal{F}$-free} if it does not contain any member of $\mathcal{F}$ as a subgraph. 
The \textbf{Tur\'{a}n number} $\mathrm{ex}(n, \mathcal{F})$ of $\mathcal{F}$ is the maximum number of edges in an $\mathcal{F}$-free $r$-graph on $n$ vertices. 
The \textbf{Tur\'{a}n density} of $\mathcal{F}$ is defined as $\pi(\mathcal{F})\coloneq \lim_{n\to\infty}\mathrm{ex}(n,\mathcal{F})/{n\choose r}$. 
A family $\mathcal{F}$ of $r$-graphs is called \textbf{degenerate} if $\pi(\mathcal{F}) = 0$. 
According to a theorem of Erd\H{o}s~\cite{Erdos64}, this is equivalent to saying that $\mathcal{F}$ contains at least one $r$-partite $r$-graph. 
Determining the growth rate of $\mathrm{ex}(n, \mathcal{F})$ for degenerate families is a central and notoriously difficult topic in Extremal Combinatorics, and it remains open for most families.
For example, the Even Cycle Problem, proposed by Erd\H{o}s~\cite{E64,BS74}, asks for the exponent of $\mathrm{ex}(n,C_{2k})$ is open for every $k$ not in $\{2,3,5\}$ (see e.g.~\cite{ERS66,Ben66,Wen91,LU93,LUW99}).  
Here, we refer the reader to the survey~\cite{FS13} for more results on degenerate Tur\'{a}n problems. 

The key property we investigate in this work is defined as follows.
\begin{definition}\label{DEF:boundedness}
      Let $\alpha, \beta > 0$ be two real numbers. 
      A family $\mathcal{F}$ of $r$-graphs is \textbf{$(\alpha, \beta)$-bounded} if there exists $N_0$ such that every $r$-graph $\mathcal{H}$ on $n\ge N_0$ vertices with
          \begin{align}\label{equ:DEF:boundedness}
              \Delta(\mathcal{H}) 
              \ge \alpha\binom{n-1}{r-1} 
              \quad\text{and}\quad
              |\mathcal{H}| 
              \ge (1-\beta) \cdot \mathrm{ex}(n,F)  
          \end{align}
          contains a member in $\mathcal{F}$ as a subgraph.  
          We say $\mathcal{F}$ is \textbf{bounded}\footnote{It is worth noting that our results apply to the stronger definition$\colon$ $\mathcal{F}$ is bounded if, for every $\alpha>0$, there exist $\beta>0$ and $N_0$ such that~\eqref{equ:DEF:boundedness} holds for all $n \ge N_0$.} if it is $(\alpha, \beta)$-bounded for some constants $0 < \alpha, \beta < 1$.
\end{definition}
Equivalently, a family $\mathcal{F}$ of $r$-graphs is $(\alpha, \beta)$-bounded if, for sufficiently large $n$, every $n$-vertex $\mathcal{F}$-free $r$-graph with $\Delta(\mathcal{H}) \ge \alpha\binom{n-1}{r-1}$ has fewer than $(1-\beta) \cdot \mathrm{ex}(n,F)$ edges.
In particular, this implies that every $n$-vertex $F$-free extremal construction cannot have a vertex of degree greater than $\alpha\binom{n-1}{r-1}$ when $n$ is large. 

Boundedness was introduced in recent works~\cite{HLLYZ23,HHLLYZ23a} that aim to extend the classical Corr\'{a}di--Hajnal Theorem~\cite{CH63} and Hajnal--Szemer\'{e}di Theorem~\cite{HS70} to a density version. 
It plays a crucial role in determining the exact bound of $\mathrm{ex}(n,tF)$, the maximum number of edges in an $n$-vertex $r$-graph with at most $(t-1)$ vertex-disjoint copies of $F$, for both nondegenerate and degenerate $r$-graphs $F$. 
Very recently, applications of boundedness in anti-Ramsey type problems, a topic initiated by Erd\H{o}s--Simonovits--S\'{o}s~\cite{ESS75}, were shown in~\cite{DHLY24}. 

In this work, we initiate the study of the boundedness of degenerate hypergraphs. 
In Theorem~\ref{THM:weak-vtx-imply-boundedness}, we present a general sufficient condition for a graph to be bounded.   In Theorem~\ref{THM:Kst-expasion-bounded}, we show that the expansion of the complete bipartite graphs, firstly studied by Kostochka--Mubayi--Verstra\"{e}te~\cite{KMV15}, are bounded in most cases. 
There are many natural classes of degenerate hypergraphs, such as complete $r$-partite $r$-graphs, where establishing boundedness remains an open problem (see discussions in Section~\ref{SEC:Remark}). 
We hope our work will motivate further research on this topic. 

\subsection{Graphs}
Given a bipartite graph $F$, we say a bipartition $V(F) = V_1 \cup V_2$ is \textbf{proper} if every edge in $F$ has nonempty intersection with both $V_1$ and $V_2$. 
For a bipartite graph $F$ with a proper bipartition $V(F) = V_1 \cup V_2$, we use $F[V_1, V_2]$ to emphasize this bipartite structure and to specify the ordering of the two sets $V_1$ and $V_2$. 
For a vertex $v\in V_1$ (resp. $v\in V_2$), we denote by $F[V_1, V_2]-v$ the bipartite subgraph obtained from $F$ by removing the vertex $v$ (and all edges incident to $v$), while preserving the ordering of the two sets $V_1\setminus\{v\}, V_2$ (resp. $V_1, V_2\setminus\{v\}$). 
For simplicity, we consider the $s$ by $t$ complete bipartite graph $K_{s,t}$ as $K_{s,t}[V_1,V_2]$, where $V_1\cup V_2 = V(K_{s,t})$ is the proper bipartition with $\left(|V_1|,|V_2|\right) = (s,t)$. 

Given a bipartite graph $F[V_1, V_2]$, we say another bipartite graph $G[U_1, U_2]$ is \textbf{ordered-$F[V_1, V_2]$-free} if there is no copy of $F$ in $G$ with $V_1 \subset U_1$ and $V_2 \subset U_2$. 
Following the definition of Zarankiewicz~\cite{Zaran51}, for integers $m,n \ge 1$, the \textbf{Zarankiewicz number} $Z(m,n,F[V_1, V_2])$ is the maximum number of edges in an ordered-$F[V_1, V_2]$-free bipartite graph $G[U_1, U_2]$ with $\left(|U_1|, |U_2|\right) = (m,n)$.   
\begin{definition}\label{DEF:weak-vertex}
    Let $F$ be a bipartite graph. 
    A vertex $v\in V(F)$ is \textbf{critical} if there exists a proper bipartition $V(F) = V_1 \cup V_2$ such that  
    \begin{align}\label{equ:DEF:weak-vertex}
        \lim_{n\to \infty} \frac{Z(n,n,F[V_1, V_2]-v)}{\mathrm{ex}(n,F)} = 0. 
    \end{align}
    %
\end{definition}
\textbf{Remark.}
Note that~\eqref{equ:DEF:weak-vertex} implies  that $\lim_{n\to \infty} \frac{\mathrm{ex}(n,F-v)}{\mathrm{ex}(n,F)} = 0$ (see Fact~\ref{FACT:ex-vs-Z}), where $F-v$ denotes the graph obtained from $F$ by removing the vertex $v$. 
This is the case where Simonovits refers to $v$ as a \textbf{weak} vertex of $F$ in~\cite{Sim84} (see e.g.~\cite{Ma17} for an application). 
The definitions of critical and weak vertices are equivalent if a conjecture of Erd\H{o}s--Simonovits (see~{\cite[Conjecture~2.12]{FS13}}), which states that $Z(n,n,F) = O\left(\mathrm{ex}(n,F)\right)$ holds for every bipartite graph $F$, is true. 

The following theorem presents a sufficient condition for a graph to be bounded. 
\begin{theorem}\label{THM:weak-vtx-imply-boundedness}
    Let $F$ be a bipartite graph that contains a cycle. 
    If $F$ contains a critical vertex $v$ such that $F-v$ is connected, then $F$ is bounded. 
\end{theorem}
Theorem~\ref{THM:weak-vtx-imply-boundedness}, together with established results on graph Zarankiewicz problems, leads to the following corollary.
%
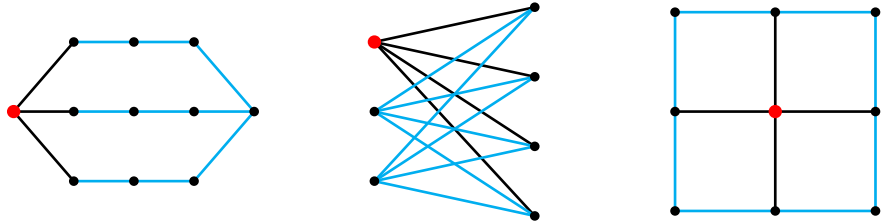
\begin{figure}[htbp]
\centering
\tikzset{every picture/.style={line width=1pt}} 
\begin{tikzpicture}[x=0.75pt,y=0.75pt,yscale=-1,xscale=1]
\draw    (120,125) -- (150,125) ;
\draw  [color=cyan]  (150,125) -- (180,125) ;
\draw  [color=cyan]  (180,125) -- (210,125) ;
\draw [color=cyan]   (210,125) -- (240,125) ;
%
\draw    (120,125) -- (150,90) ;
\draw [color=cyan]   (150,90) -- (180,90) ;
\draw [color=cyan]   (180,90) -- (210,90) ;
\draw [color=cyan]   (210,90) -- (240,125) ;
%
\draw    (120,125) -- (150,160) ;
\draw  [color=cyan]  (150,160) -- (180,160) ;
\draw  [color=cyan]  (180,160) -- (210,160) ;
\draw  [color=cyan]  (210,160) -- (240,125) ;
\draw [red, fill = red] (120,125) circle (2pt);
\draw [fill=black] (150,125) circle (1.3pt);
\draw [fill=black] (180,125) circle (1.3pt);
\draw [fill=black] (210,125) circle (1.3pt);
\draw [fill=black] (240,125) circle (1.3pt);
\draw [fill=black] (150,90) circle (1.3pt);
\draw [fill=black] (180,90) circle (1.3pt);
\draw [fill=black] (210,90) circle (1.3pt);
\draw [fill=black] (150,160) circle (1.3pt);
\draw [fill=black] (180,160) circle (1.3pt);
\draw [fill=black] (210,160) circle (1.3pt);
%
%
%
%
\draw    (300,90) -- (380,72.5) ;
\draw    (300,90) -- (380,107.5) ;
\draw    (300,90) -- (380,142.5) ;
\draw    (300,90) -- (380,177.5) ;
%
\draw  [color=cyan]  (300,125) -- (380,72.5) ;
\draw  [color=cyan]  (300,125) -- (380,107.5) ;
\draw   [color=cyan] (300,125) -- (380,142.5) ;
\draw   [color=cyan] (300,125) -- (380,177.5) ;
\draw  [color=cyan]  (300,160) -- (380,72.5) ;
\draw  [color=cyan]  (300,160) -- (380,107.5) ;
\draw   [color=cyan] (300,160) -- (380,142.5) ;
\draw  [color=cyan]  (300,160) -- (380,177.5) ;
\draw [red, fill = red] (300,90) circle (2pt);
\draw [fill=black] (300,125) circle (1.3pt);
\draw [fill=black] (300,160) circle (1.3pt);
\draw [fill=black] (380,72.5) circle (1.3pt);
\draw [fill=black] (380,107.5) circle (1.3pt);
\draw [fill=black] (380,142.5) circle (1.3pt);
\draw [fill=black] (380,177.5) circle (1.3pt);
%
%
%
%
%
\draw  [color=cyan]  (450,75) -- (550,75) ;
\draw    (450,125) -- (550,125) ;
\draw  [color=cyan]  (450,175) -- (550,175) ;
\draw  [color=cyan]  (450,75) -- (450,175) ;
\draw    (500,75) -- (500,175) ;
\draw  [color=cyan]  (550,75) -- (550,175) ;
\draw [fill=black] (450,75) circle (1.3pt);
\draw [fill=black] (550,75) circle (1.3pt);
\draw [fill=black] (500,75) circle (1.3pt);
\draw [fill=black] (450,175) circle (1.3pt);
\draw [fill=black] (500,175) circle (1.3pt);
\draw [fill=black] (550,175) circle (1.3pt);
\draw [fill=black] (450,125) circle (1.3pt);
\draw [red, fill = red] (500,125) circle (2pt);
\draw [fill=black] (550,125) circle (1.3pt);
\end{tikzpicture}
\caption{The theta graph $\Theta_{4,4,4}$, the complete bipartite graph $K_{3,4}$, and the $2\times 2$ grid.}
\label{Fig:bipartite-graphs}
\end{figure}
%
\begin{corollary}\label{CORO:bounded-graphs}
    The following bipartite graphs are bounded. 
    \begin{enumerate}[label=(\roman*)]
        \item\label{CORO:bounded-graphs-1} All non-forest bipartite graphs that become a tree after the removal of a vertex. This includes even cycles $C_{2k}$ for $k \ge 2$ and all bipartite theta graphs. 
        \item\label{CORO:bounded-graphs-2} The complete bipartite graph $K_{s,t}$ with $t > \min\left\{s^2 -3s +3, (s-1)!\right\}$.  
        \item\label{CORO:bounded-graphs-3} The $2$ by $2$ grid. 
    \end{enumerate}
\end{corollary}

Proofs for Theorem~\ref{THM:weak-vtx-imply-boundedness} and Corollary~\ref{CORO:bounded-graphs} are presented in Section~\ref{SEC:proof-graph-bounded}. 

\subsection{Expansion of bipartite graphs}
Given a graph $F$, the \textbf{expansion} $F^{+}$ of $F$ is the $3$-graph obtained by adding a new vertex to each edge of $F$, ensuring that different edges receive the different vertices.
We call this graph $F$ the \textbf{core} of $F^{+}$. 

In~\cite{KMV15}, Kostochka--Mubayi--Verstra\"{e}te  studied the Tur\'{a}n problem concerning the expansion of the complete bipartite graphs and established the following bounds.
\begin{theorem}[{\cite[Theorem~1.4]{KMV15}}]\label{THM:KMV-upper-Kst+}
    Suppose that $t \ge s \ge 3$ are integers. 
    Then 
    \begin{align*}
        \mathrm{ex}(n,K_{s,t}^{+}) 
        = O\left(n^{3-\frac{3}{s}}\right). 
    \end{align*}
    Moreover, if $t > (s-1)! \ge  2$, then $\mathrm{ex}(n,K_{s,t}^{+}) = \Omega\left(n^{3-\frac{3}{s}}\right)$. 
\end{theorem}
In the following theorem, we establish the boundedness of the expansion of the complete bipartite graphs. 
\begin{theorem}\label{THM:Kst-expasion-bounded}
    The $3$-graph $K_{s,t}^{+}$ is bounded for all integers $s,t$ satisfying $s\ge 4$ and $t > \min\left\{\frac{3}{2}s^2 - \frac{21}{4}s + \frac{57}{8} + \frac{3}{8(2s-3)}, (s-1)!\right\}$. 
\end{theorem}
\textbf{Remark.}
It follows from~{\cite[Theorem~1.2]{KMV17II}} that $K_{2,t}^{+}$ is not bounded for any $t \ge 2$. 
While our proof of Theorem~\ref{THM:Kst-expasion-bounded} could potentially be adapted to show that $K_{3,t}^{+}$ is bounded for large $t$, we have not explored this case in general. 

A key ingredient in the proof of Theorem \ref{THM:Kst-expasion-bounded} is the following Zarankiewicz-type extension of the theorem by Kostochka--Mubayi--Verstra\"{e}te, which might be of independent interest. 

A $3$-graph $\mathcal{H}$ is \textbf{semibipartite} if there exists a bipartition $V(\mathcal{H}) = V_1 \cup V_2$ such that every edge in $\mathcal{H}$ contains exact two vertices from $V_1$. 
Similar to the graph case, we use $\mathcal{H}[V_1, V_2]$ to emphasize this semibipartite structure and specify the ordering of $V_1$ and $V_2$. 
We say $\mathcal{H}[V_1, V_2]$ contains an \textbf{ordered} copy of $K_{s,t}^{+}$ if there is a copy of $K_{s,t}^{+}$ in $\mathcal{H}$ such that the size-$s$ part of its core is contained in $V_1$ and the size-$t$ part of its core is contained in $V_2$.

Given positive integers $m, n, s_1, s_2, t_1, t_2$, we use $Z(m,n,K_{s_1,t_1}^{+}, K_{s_2,t_2}^{+})$ to denote the maximum number of edges in a semibipartite $3$-graph $\mathcal{H}[V_1, V_2]$ subject to the following constraints$\colon$
\begin{enumerate}[label=(\roman*)]
    \item\label{assump:Zaran-expansion-1} $\left(|V_1|, |V_2|\right) = (m,n)$, 
    \item\label{assump:Zaran-expansion-2} there is no ordered copy of $K_{s_1,t_1}^{+}$ in $\mathcal{H}[V_1, V_2]$, and 
    \item\label{assump:Zaran-expansion-3} there is no copy of $K_{s_2,t_2}^{+}$ in $\mathcal{H}$ whose core is completely contained in $V_1$.  
\end{enumerate}
%
%
\begin{theorem}\label{THM:expansion-Zaran}
    Suppose that $m, n \ge 1$, $t_1 \ge  s_1 \ge  2$, and $t_2 \ge  s_2 \ge  2$ are integers. Then
    \begin{enumerate}[label=(\roman*)]
        \item\label{THM:expansion-Zaran-1} $Z(m,n,K_{s_1,t_1}^{+}, K_{s_2,t_2}^{+}) = O\left(m^{2-\frac{1}{s_2}} n^{1-\frac{2}{s_1}} + mn+ m^2+ n^{1+\frac{1}{s_1}}\right)$.
        In particular, 
        \begin{align*}
            Z(n,n,K_{s_1,t_1}^{+}, K_{s_2,t_2}^{+}) = O\left(n^{3-\frac{1}{s_2}-\frac{2}{s_1}}\right). 
        \end{align*}
        \item\label{THM:expansion-Zaran-2} $Z(m,n,K_{t_1,s_1}^{+}, K_{s_2,t_2}^{+}) = O\left(m^{2-\frac{1}{s_1}-\frac{2}{s_2}+\frac{1}{s_1 s_2}} n + mn + m^2 + m^{2+\frac{1}{s_1} - \frac{2}{s_2}}\right)$. 
        In particular, 
        \begin{align*}
            Z(n,n,K_{t_1,s_1}^{+}, K_{s_2,t_2}^{+}) = O\left(n^{3-\frac{1}{s_2}-\frac{2}{s_1}+\frac{1}{s_1 s_2}}\right). 
        \end{align*}
        \item\label{THM:expansion-Zaran-3} If $t_1 > (s_1 - 1)! \ge  2$ and $t_2 > (s_2 -1)! \ge  2$, then 
        \begin{align*}
            \min\left\{Z(n,n,K_{s_1,t_1}^{+}, K_{s_2,t_2}^{+}),\ Z(n,n,K_{t_1,s_1}^{+}, K_{s_2,t_2}^{+})\right\}
             =  \Omega\left( n^{3- \frac{1}{s_2}-\frac{2}{s_1}} \right).  
        \end{align*}
    \end{enumerate}
\end{theorem}
\textbf{Remarks.}
\begin{itemize}
    \item Let $\mathcal{H}$ be a $2n$-vertex $K_{s,t}^{+}$-free $3$-graph with $\mathrm{ex}(2n,K_{s,t}^{+})$ edges. 
    It follows from a standard probabilistic argument that there exists a balanced bipartition $V(\mathcal{H}) = V_1 \cup V_2$ such that the induced $n$ by $n$ semibipartite subgraph $\mathcal{H}[V_1, V_2]$ contains at least $(3/8-o(1))|\mathcal{H}|$ edges. 
    Note that $\mathcal{H}[V_1, V_2]$ is ordered-$K_{s,t}^{+}$-free and does not contain a copy of $K_{s,t}^{+}$ with the core contained in $V_1$.  
    Therefore, we have $Z(n,n,K_{s,t}^{+},K_{s,t}^{+}) \ge |\mathcal{H}[V_1, V_2]| \ge (3/8-o(1)) \cdot \mathrm{ex}(2n,K_{s,t}^{+})$. 
    Combined with Theorem~\ref{THM:expansion-Zaran}~\ref{THM:expansion-Zaran-1} and simple calculations, we obtain that $\mathrm{ex}(2n,K_{s,t}^{+}) = O\left(n^{3-\frac{3}{s}}\right)$, which implies the upper bound in Theorem~\ref{THM:KMV-upper-Kst+}.
    \item The constraint $t_2 > (s_2 -1)!$ in Theorem~\ref{THM:expansion-Zaran}~\ref{THM:expansion-Zaran-3} can be relaxed to $t_2 > 9^{s_2}\cdot s_{2}^{4s_2^{2/3}}$ using a recently breakthrough result by Bukh~{\cite[Theorem~1]{Bukh21}}. 
\end{itemize}

We will present the proof of Theorem~\ref{THM:Kst-expasion-bounded}, assuming Theorem~\ref{THM:expansion-Zaran}, in Section~\ref{SEC:proof-Kst-expasion}. 
The proof of Theorem~\ref{THM:expansion-Zaran} will be presented in Section~\ref{SEC:proof-Zaran-expansion}. 

\section{Preliminaries}
We present some definitions and preliminary results in this section. 

For a graph $G$ and two disjoint sets $S, T \subset V(G)$, the \textbf{induced bipartite subgraph} $G[S,T]$ is the collection of edges in $G$ that connect vertices between $S$ and $T$. 
The \textbf{induced subgraph} $G[S]$ is the collection of edges in $G$ that are completely contained in $S$. 
For a vertex $v\in V(G)$, the \textbf{neighborhood} of $v$ in $G$ is $N_{G}(v) \coloneqq \left\{u\in V(G) \colon \{u,v\} \in G\right\}$. 
The subscript $G$ will be omitted if it is clear from the context. 

The following fact follows from a minor modification of the proof for~{\cite[Corollary~2.15]{FS13}}. 
\begin{fact}\label{FACT:ex-vs-Z}
    For every bipartite graph $F[V_1, V_2]$ and for every $n \in \mathbb{N}$, 
    \begin{align*}
        \frac{1}{2} \cdot \mathrm{ex}(2n,F)
        \le Z(n,n,F[V_1, V_2]). 
    \end{align*}
\end{fact}
The following result of Erd\H{o}s~\cite{PER61} follows from a standard probablistic argument.
\begin{theorem}[\cite{PER61}]\label{THM:Erdos-lower-bound}
    Suppose that $F$ is a bipartite graph that contains a cycle. 
    Then there exists a constant $\gamma> 0$ such that 
    \begin{align*}
        \mathrm{ex}(n,F)
        = \Omega\left(n^{1+\gamma}\right). 
    \end{align*}
\end{theorem}
The bound established by Erd\H{o}s can be improved in certain cases.  
In particular, a celebrated result of Alon--R\'{o}nyai--Szab\'{o}~\cite{ARS99} is as follows. 
\begin{theorem}[\cite{ARS99}]\label{THM:ARS99}
    Suppose that $s \ge 2$ and $t > (s-1)!$ are integers. 
    Then 
    \begin{align*}
        \mathrm{ex}(n,K_{s,t}) = \Omega\left(n^{2-\frac{1}{s}}\right). 
    \end{align*}
\end{theorem}
The following two theorems, by K\"{o}vari--S\'{o}s--Tur\'{a}n~\cite{KST54} and Naor--Verstra\"ete~\cite{NV05}, respectively, concerning graph Zarankiewicz problems, will be useful. 
\begin{theorem}[\cite{KST54}]\label{THM:KST54}
    Let $m,n,s,t \ge 1$ be integer. 
    Then 
    \begin{align*}
        \mathrm{ex}(n,K_{s,t}) 
        & \le \frac{\left(t-1\right)^{\frac{1}{s}}}{2}n^{2-\frac{1}{s}}+ \frac{s-1}{2}n, \quad\text{and}\quad \\
        Z(m,n,K_{s,t})
        & \le \left(t-1\right)^{\frac{1}{s}} m n^{1-\frac{1}{s}} + (s-1)n. 
    \end{align*}
    In particular, $Z(n,n,K_{s,t}) \le \left(t-1\right)^{\frac{1}{s}} n^{2-\frac{1}{s}} + (s-1)n$. 
\end{theorem}
\begin{theorem}[\cite{NV05}]\label{THM:NV-even-cycle}
    Let $m,n,k \ge 2$ be integers. Then 
    \begin{align*}
        Z(m,n,C_{2k})
        \le 
        \begin{cases}
            (2k-3) \left(m^{\frac{1}{2}+\frac{1}{2k}} n^{\frac{1}{2}+\frac{1}{2k}} + m+n\right), & \quad\text{if}\quad k \equiv 1 \pmod{2}, \\
            (2k-3) \left(m^{\frac{1}{2}+\frac{1}{k}} n^{\frac{1}{2}} + m+n\right), & \quad\text{if}\quad k \equiv 0 \pmod{2}. 
        \end{cases}
    \end{align*}
    In particular, $Z(n,n,C_{2k}) \le (2k-3) \left(n^{1+\frac{1}{k}} + 2n\right)$. 
\end{theorem}

Given an $r$-graph $\mathcal{H}$ and an integer $1\le i \le r-1$, the \textbf{$i$-th shadow} of $\mathcal{H}$ is 
\begin{align*}
    \partial_i \mathcal{H}
    \coloneqq \left\{e\in \binom{V(\mathcal{H})}{r-i} \colon \text{there exists $E\in \mathcal{H}$ such that $e\subset E$}\right\}. 
\end{align*}
For convenience, we set $\partial \mathcal{H} \coloneqq \partial_1 \mathcal{H}$. 
For an $i$-set $T \subset V(\mathcal{H})$ the \textbf{link} of $T$ in $\mathcal{H}$ is 
\begin{align*}
    L_{\mathcal{H}}(T)
    \coloneqq \left\{e\in \binom{V(\mathcal{H})}{r-i} \colon e\cup T \in \mathcal{H}\right\}, 
\end{align*}
and the \textbf{degree} $d_{\mathcal{H}}(T)$ of $T$ in $\mathcal{H}$ is the size of $L_{\mathcal{H}}(T)$. 

The following fact follows from a simple greedy argument. 
\begin{fact}\label{FACT:expansion-free-shadow}
    Suppose that $t \ge s \ge 1$ are integers and $\mathcal{H}$ is $3$-graph.  
    Then every copy of $K_{s,t}$ in the set $\left\{e\in \partial\mathcal{H} \colon d_{\mathcal{H}}(e) \ge st+s+t\right\}$ can be extended to a copy of $K_{s,t}^{+}$ in $\mathcal{H}$. 
\end{fact}

We say an $r$-graph $F$ is \textbf{connected} if the graph $\partial_{r-2}F$ is a connected. 
The following simple inequality on the Tur\'{a}n numbers of connected $r$-graphs will be useful. 
\begin{proposition}\label{PROP:ex-increasing}
       Suppose that $\mathcal{F}$ is a family of connected $r$-graphs. 
       Then 
       \begin{align*}
           \mathrm{ex}(m,\mathcal{F}) 
           \le \left(1 - \left(\frac{n-m-r}{n}\right)^{r}\right) \cdot \mathrm{ex}(n,\mathcal{F}). 
       \end{align*}
\end{proposition}
\begin{proof}[Proof of Proposition~\ref{PROP:ex-increasing}]
    Let $\mathcal{F}$ be a family of connected $r$-graphs. 
    A result of Katona--Nemetz--Simonovits~\cite{KNS64}, which follows from a simple averaging argument, states that $\mathrm{ex}(n,\mathcal{F})/\binom{n}{r}$ is decreasing in $n$. 
    Therefore, 
    \begin{align*}
        \mathrm{ex}(n-m,\mathcal{F})
        \ge \frac{\binom{n-m}{r}}{\binom{n}{t}} \cdot \mathrm{ex}(n,\mathcal{F})
        \ge \left(\frac{n-m-r}{n}\right)^{r} \cdot \mathrm{ex}(n,\mathcal{F}). 
    \end{align*} 
    Let $\mathcal{H}_{1}$ be an $m$-vertex $\mathcal{F}$-free $r$-graph with exactly $\mathrm{ex}(m,\mathcal{F})$ edges, and let $\mathcal{H}_{2}$ be an $(n-m)$-vertex $\mathcal{F}$-free $r$-graph with exactly $\mathrm{ex}(n-m,\mathcal{F})$ edges. 
    Define $\mathcal{H}$ as the vertex-disjoint union of $\mathcal{H}_1$ and $\mathcal{H}_2$. 
    Since every member in $\mathcal{F}$ is connected, $\mathcal{H}$ is $\mathcal{F}$-free. 
    Hence, 
    \begin{align*}
        \mathrm{ex}(n,\mathcal{F})
        \ge |\mathcal{H}|
        = \mathrm{ex}(m,\mathcal{F}) + \mathrm{ex}(n-m,\mathcal{F})
        \ge \mathrm{ex}(m,\mathcal{F}) + \left(\frac{n-m-r}{n}\right)^{r} \cdot \mathrm{ex}(n,\mathcal{F}), 
    \end{align*}
    which implies Proposition~\ref{PROP:ex-increasing}. 
\end{proof}

\section{Proofs of Theorem~\ref{THM:weak-vtx-imply-boundedness} and Corollary~\ref{CORO:bounded-graphs}}\label{SEC:proof-graph-bounded}
In this section, we prove Theorem~\ref{THM:weak-vtx-imply-boundedness} and Corollary~\ref{CORO:bounded-graphs}. 
First, let us present the proof of Theorem~\ref{THM:weak-vtx-imply-boundedness}. 
\begin{proof}[Proof of Theorem~\ref{THM:weak-vtx-imply-boundedness}]
    Let $F$ be a bipartite graph that contains a cycle, and assume that $v_{\ast} \in V(F)$ is a critical vertex such that $\widetilde{F} \coloneqq F-v_{\ast}$ is connected.
    Since $v_{\ast}$ is a critical vertex of $F$, $v_{\ast}$ cannot be an isolated vertex in $F$ (otherwise we would have $\mathrm{ex}(n,\widetilde{F}) = \mathrm{ex}(n,F)$ for all $n \ge v(F)$). 
    Combined with the assumption that $\widetilde{F}$ is connected, we know that $F$ is connected as well.
    Hence, there is a unique proper bipartition $U_1 \cup U_2 = V(F)$ of $F$. 
    By symmetry, we may assume that $v_{\ast} \in U_2$. 
    Let $W_1 \coloneqq U_1$ and $W_2 \coloneqq U_2\setminus\{v_{\ast}\}$. 
    Since $\widetilde{F}$ is connected, $W_1 \cup W_2 = V(\widetilde{F})$ is the  unique proper bipartition of $\widetilde{F}$. 
    
    Let $\alpha \in (0,1)$ be a real number and $n$ be a sufficiently large integer. 
    Let $G$ be an $n$-vertex $F$-free graph with maximum degree $\Delta \ge \alpha n$.
    Fix a vertex $v\in V(G)$ with $d_{G}(v) = \Delta$. 
    Let $V_1 \coloneqq N_{G}(v)$ and $V_2 \coloneqq V(G) \setminus \left(N_{G}(v) \cup \{v\}\right)$. 
    Notice that 
    \begin{itemize}
        \item The induced subgraph $G[V_1]$ is $F^{-}$-free, where $F^{-}$ is the collection of graphs obtained from $F$ by removing a vertex. 
        \item The induced bipartite subgraph $G[V_1, V_2]$ is ordered-$\widetilde{F}[W_1, W_2]$-free. 
    \end{itemize}
    Therefore, we obtain 
    \begin{align}\label{equ:gp-G-upper-bound}
        |G|
        & \le \Delta + \mathrm{ex}(\Delta, F^{-}) 
            + Z(\Delta, n-\Delta-1, \widetilde{F}[W_1, W_2]) 
            + \mathrm{ex}(n-\Delta-1, F) \notag \\
        & \le n + \mathrm{ex}(n, F^{-}) 
            + Z(n, n, \widetilde{F}[W_1, W_2]) 
            + \mathrm{ex}((1-\alpha)n, F). 
    \end{align}
    Since $F$ contains a cycle, it follows from Theorem~\ref{THM:Erdos-lower-bound} that 
    \begin{align*}
        n = o\left(\mathrm{ex}(n,F)\right). 
    \end{align*}
    Since $v_{\ast}$ is a critical vertex, it follows from the definition that 
    \begin{align*}
        Z(n, n, \widetilde{F}[W_1, W_2]) 
        = o\left(\mathrm{ex}(n,F)\right). 
    \end{align*}
    In particular, by Fact~\ref{FACT:ex-vs-Z}, 
    \begin{align*}
        \mathrm{ex}(n, F^{-})
        \le \mathrm{ex}(n, \widetilde{F})
        \le 2\cdot Z(n, n, \widetilde{F}[W_1, W_2])
        = o\left(\mathrm{ex}(n,F)\right).
    \end{align*}
    Finally, it follows from Proposition~\ref{PROP:ex-increasing} that 
    \begin{align*}
        \mathrm{ex}((1-\alpha)n, F)
        \le \left(1 - \left(\frac{\alpha n - r}{n}\right)^2\right) \cdot \mathrm{ex}(n,F)
        \le \left(1 - \frac{\alpha^2}{2} \right) \cdot \mathrm{ex}(n,F). 
    \end{align*}
    Therefore, Inequality~\eqref{equ:gp-G-upper-bound} continues as 
    \begin{align*}
        |G| 
        \le 3 \cdot o\left(\mathrm{ex}(n,F)\right) + \left(1 - \frac{\alpha^2}{2} \right) \cdot \mathrm{ex}(n,F) 
        \le \left(1 - \frac{\alpha^2}{3} \right) \cdot \mathrm{ex}(n,F), 
    \end{align*}
    which proves Theorem~\ref{THM:weak-vtx-imply-boundedness}. 
\end{proof}
Next, we prove Corollary~\ref{CORO:bounded-graphs}. 
\begin{proof}[Proof of Corollary~\ref{CORO:bounded-graphs}]
    Let $F$ be a non-forest bipartite graph and $v\in V(F)$ be a vertex such that $F-v$ is a tree (in particular, $F-v$ is connected). 
    A simple greedy argument shows that $Z(n,n,F-v) = O(n)$ (see e.g. the proof of~{\cite[Theorem~2.32]{FS13}}). 
    On the other hand, since $F$ is not a forest, it follows from Theorem~\ref{THM:Erdos-lower-bound} that $\mathrm{ex}(n,F) = \Omega(n^{1+\gamma})$ for some constant $\gamma>0$. 
    In particular, this implies that $v$ is critical. 
    So it follows from Theorem~\ref{THM:weak-vtx-imply-boundedness} that $F$ is bounded. 
    This proves Corollary~\ref{CORO:bounded-graphs}~\ref{CORO:bounded-graphs-1}. 

    Corollary~\ref{CORO:bounded-graphs}~\ref{CORO:bounded-graphs-2} follows easily from Theorem~\ref{THM:KST54} on the upper bound of $Z(n,n,K_{s-1,t})$, Theorem~\ref{THM:ARS99} the lower bound of $\mathrm{ex}(n,K_{s,t})$, and the standard probabilistic lower bound $\mathrm{ex}(n,K_{s,t}) = \Omega\left(n^{2-\frac{s+t-2}{st-1}}\right)$ (see e.g.~\cite{ES74}). 

    Let $F$ be the $2$ by $2$ grid. 
    Since $F$ contains $C_4$ as a subgraph, it follows from the well-known construction of Erd\H{o}s--R\'{e}nyi~\cite{ER62} that $\mathrm{ex}(n,F) \ge \mathrm{ex}(n,C_4) = (1/2-o(1))n^{3/2}$. 
    On the other hand, notice that the graph obtained from $F$ by removing the center vertex is $C_8$ (see Figure~\ref{Fig:bipartite-graphs}). According to Theorem~\ref{THM:NV-even-cycle}, $Z(n,n,C_8) = O(n^{5/4})$. 
    So it follows from Theorem~\ref{THM:weak-vtx-imply-boundedness} that $F$ is bounded. This proves Corollary~\ref{CORO:bounded-graphs}~\ref{CORO:bounded-graphs-3}. 
\end{proof}

\section{Proof of Theorem~\ref{THM:Kst-expasion-bounded}}\label{SEC:proof-Kst-expasion}
In this section, we establish the boundedness of $K_{s,t}^{+}$, assuming Theorem~\ref{THM:expansion-Zaran}. 
The following bounds established by Kostochka--Mubayi--Verstra\"{e}te~\cite{KMV15} will be useful. 
\begin{proposition}[{\cite[Proposition~1.1]{KMV15}}]\label{PROP:KMV-hypergraph-random-lower}
    For all integers $t \ge  s \ge  2$, 
    \begin{align*}
      \Omega\left(n^{3-\frac{3s+3t-9}{st-3}}\right) 
      = \mathrm{ex}(n,K_{s,t}^{+}) 
      = O\left(n^{3-\frac{1}{s}}\right).
    \end{align*}
\end{proposition}
%
%
\begin{proof}[Proof of Theorem~\ref{THM:Kst-expasion-bounded}]
    Let $\alpha \in (0,1)$ be a real number, $t \ge s\ge 4$ be integers satisfying
    \begin{align*}
        t 
        > 
        \begin{cases}
            (s-1)!, & \quad\text{if}\quad s = 4, \\
            \frac{3}{2}s^2 - \frac{21}{4}s + \frac{57}{8} + \frac{3}{8(2s-3)}, & \quad\text{if}\quad s \ge 5. 
        \end{cases}
    \end{align*}
    Let $n$ be a sufficiently large integer. Note that, according to Theorem~\ref{THM:KMV-upper-Kst+} and Proposition~\ref{PROP:KMV-hypergraph-random-lower}, the choice of $t$ ensures that 
    \begin{align}\label{equ:expansion-bounded-ex-n-lower-bound}
        \mathrm{ex}(n,K_{s,t}^{+})
        = 
        \begin{cases}
            \Omega\left(n^{3-\frac{3}{s}}\right), & \quad\text{if}\quad s = 4, \\
            \Omega\left(n^{3-\frac{3s+3t-9}{st-3}}\right), & \quad\text{if}\quad s \ge 5. 
        \end{cases}
    \end{align}
    In addition, simple calculations show that 
    \begin{align}\label{equ:expansion-bounded-simple-inequ}
        \max\left\{3-\frac{3}{s-1}, 3-\frac{3s-4}{(s-1)^2}, 3-\frac{3s-1}{s(s-1)}\right\}
        < \begin{cases}
            3-\frac{3}{s}, & \quad\text{if}\quad s = 4, \\
            3-\frac{3s+3t-9}{st-3}, & \quad\text{if}\quad s \ge 5. 
        \end{cases}
    \end{align}
    Let $\mathcal{H}$ be an $n$-vertex $K_{s,t}^{+}$-free $3$-graph with maximum degree $\Delta \ge \alpha \binom{n-1}{2}$. 
    Let $v \in V(\mathcal{H})$ be a vertex with degree $d_{\mathcal{H}}(v) = \Delta \ge \alpha \binom{n-1}{2}$. 
    Define 
    \begin{align*}
        V_1 \coloneqq \left\{u \in V(\mathcal{H}) \setminus \{v\} \colon d_{\mathcal{H}}(uv) \ge (s+1)(t+1)\right\}
        \quad\text{and}\quad 
        V_2 \coloneqq V(\mathcal{H}) \setminus \left(V_1 \cup \{v\}\right). 
    \end{align*}
    Note that $V_1$ is the collection of vertices that have degree at least $(s+1)(t+1)$ in the link graph $L_{\mathcal{H}}(v)$. 
    \begin{claim}\label{CLAIM:expansion-V1-lower-bound}
        We have $|V_1| \ge \frac{\alpha}{2}n$, and hence, $|V_2| \le \left(1-\frac{\alpha}{2}\right)n$. 
    \end{claim}
    \begin{proof}[Proof of Claim~\ref{CLAIM:expansion-V1-lower-bound}]
        Suppose to the contrary that $|V_1| < \frac{\alpha}{2}n$. 
        Then we would have 
        \begin{align*}
            |L_{\mathcal{H}}(v)|
            & \le \frac{|V_1|\cdot (n-1) + |V_2| \cdot (s+1)(t+1)}{2} \\
            & \le \frac{\frac{\alpha n}{2}\cdot (n-1) + (n-1) \cdot (s+1)(t+1)}{2}
            < \alpha \binom{n-1}{2}, 
        \end{align*}
        contradicting the assumption that $|L_{\mathcal{H}}(v)| = d_{\mathcal{H}}(v) \ge \alpha \binom{n-1}{2}$. 
    \end{proof}
    Since $\mathcal{H}[V_2]$ is $K_{s,t}^{+}$-free, it follows from Claim~\ref{CLAIM:expansion-V1-lower-bound} and Proposition~\ref{PROP:ex-increasing} that 
    \begin{align}\label{equ:expansion-bound-HV2}
        |\mathcal{H}[V_2]|
        \le \mathrm{ex}(|V_2|, K_{s,t}^{+}) 
         \le \mathrm{ex}((1-\alpha/2)n, K_{s,t}^{+}) 
        & \le \left(1-\left(\frac{\alpha n/2-r}{n}\right)^3\right) \cdot \mathrm{ex}(n, K_{s,t}^{+}) \notag \\
        & \le \left(1-\frac{\alpha^3}{9}\right) \cdot \mathrm{ex}(n, K_{s,t}^{+}). 
    \end{align}
    Next, we consider the upper bound for the number of edges that have nonempty intersection with $V_1$. 
    For $i\in \{1,2\}$, let $\mathcal{G}_i$ denote the collection of edges in $\mathcal{H}-v$ that contain exactly two vertices from $V_i$. 
    Note that both $\mathcal{G}_1[V_1, V_2]$ and $\mathcal{G}_2[V_2, V_1]$ are semibipartite. 
    The following claim follows from the definition of $V_1$ and a simple greedy argument (see e.g. Fact~\ref{FACT:expansion-free-shadow}). 
    \begin{claim}\label{CLAIM:expansion-3-parts-free}
        The following statements hold. 
        \begin{enumerate}[label=(\roman*)]
            \item\label{CLAIM:expansion-3-parts-free-1}
                The induced subgraph $\mathcal{H}[V_1]$ is $K_{s-1,t}^{+}$-free. 
            \item\label{CLAIM:expansion-3-parts-free-2}
                The semibipartite $3$-graph $\mathcal{G}_{1}[V_1, V_2]$ is ordered-$K_{t,s-1}^{+}$-free. 
            \item\label{CLAIM:expansion-3-parts-free-3}
                The semibipartite $3$-graph $\mathcal{G}_{2}[V_2, V_1]$ is ordered-$K_{s-1,t}^{+}$-free.
        \end{enumerate}
    \end{claim}
    %
    It follows from Claim~\ref{CLAIM:expansion-3-parts-free}~\ref{CLAIM:expansion-3-parts-free-1} and Theorem~\ref{THM:KMV-upper-Kst+} that 
    \begin{align}\label{equ:expansion-bound-HV1}
        |\mathcal{H}[V_1]|
        \le \mathrm{ex}(|V_1|, K_{s-1,t}^{+})
        \le \mathrm{ex}(n, K_{s-1,t}^{+})
        = O\left(n^{3-\frac{3}{s-1}}\right)
        = o \left(\mathrm{ex}(n, K_{s,t}^{+})\right), 
    \end{align}
    where the last equality follows from~\eqref{equ:expansion-bounded-ex-n-lower-bound} and~\eqref{equ:expansion-bounded-simple-inequ}. 

    It follows from Claim~\ref{CLAIM:expansion-3-parts-free}~\ref{CLAIM:expansion-3-parts-free-2} and Theorem~\ref{THM:expansion-Zaran}~\ref{THM:expansion-Zaran-2} that 
    \begin{align}\label{equ:expansion-bound-HV1V2}
        |\mathcal{G}_1[V_1, V_2]|
        & \le Z(|V_1|,|V_2|,K_{t,s-1}^{+},K_{s-1,t}^{+})  \notag \\
        & \le Z(n,n,K_{t,s-1}^{+},K_{s-1,t}^{+})
         =  O\left(n^{3-\frac{3s-4}{(s-1)^2}}\right) 
        = o \left(\mathrm{ex}(n, K_{s,t}^{+})\right), 
    \end{align}
    where the last equality follows from~\eqref{equ:expansion-bounded-ex-n-lower-bound} and~\eqref{equ:expansion-bounded-simple-inequ}. 

    It follows from Claim~\ref{CLAIM:expansion-3-parts-free}~\ref{CLAIM:expansion-3-parts-free-3} and Theorem~\ref{THM:expansion-Zaran}~\ref{THM:expansion-Zaran-1} that 
    \begin{align}\label{equ:expansion-bound-HV2V1}
        |\mathcal{G}_2[V_2, V_1]|
        & \le Z(|V_2|,|V_1|,K_{s-1,t}^{+},K_{s,t}^{+})  \notag \\
        & \le Z(n,n,K_{s-1,t}^{+},K_{s,t}^{+})
         =  O\left(n^{3-\frac{3s-1}{s(s-1)}}\right) 
        = o \left(\mathrm{ex}(n, K_{s,t}^{+})\right), 
    \end{align}
    where the last equality follows from~\eqref{equ:expansion-bounded-ex-n-lower-bound} and~\eqref{equ:expansion-bounded-simple-inequ}. 

    Therefore, it follows from~\eqref{equ:expansion-bound-HV2},~\eqref{equ:expansion-bound-HV1},~\eqref{equ:expansion-bound-HV1V2}, and~\eqref{equ:expansion-bound-HV2V1} that 
    \begin{align*}
        |\mathcal{H}|
        & = d_{\mathcal{H}}(v) + |\mathcal{H}[V_1]| + |\mathcal{G}_1[V_1, V_2]| + |\mathcal{G}_2[V_2, V_1]| + |\mathcal{H}[V_2]| \\
        & \le n^2 + 3\cdot o \left(\mathrm{ex}(n, K_{s,t}^{+})\right) + \left(1-\frac{\alpha^3}{9}\right) \cdot \mathrm{ex}(n, K_{s,t}^{+}) 
        < \left(1-\frac{\alpha^3}{10}\right) \cdot \mathrm{ex}(n, K_{s,t}^{+}), 
    \end{align*}
     proving Theorem~\ref{THM:Kst-expasion-bounded}.
\end{proof}

\section{Proof of Theorem~\ref{THM:expansion-Zaran}}\label{SEC:proof-Zaran-expansion}
%
\subsection{Upper bound}
In this section, we prove Theorem~\ref{THM:expansion-Zaran}~\ref{THM:expansion-Zaran-1}. 
Since the proof for Theorem~\ref{THM:expansion-Zaran}~\ref{THM:expansion-Zaran-2} is essential the same, we include it in the appendix.  
The following extension of~{\cite[Lemma~2.1]{KMV15}} will be useful.

Let $n \ge r > i \ge 1$, $d \ge 1$ be integers and $\mathcal{H}$ be an $n$-vertex $r$-graph. 
A set $\mathcal{E} \subset \binom{V(\mathcal{H})}{i}$ is \textbf{$d$-full} in $\mathcal{H}$ if, for each $e\in \mathcal{E}$, either $d_{\mathcal{H}}(e) = 0$ or $d_{\mathcal{H}}(e) \ge d$.  
\begin{lemma}\label{LEMMA:full-subgraph}
    Let $n \ge r > i \ge 1$, $\ell \ge 1$, and $d_1, \ldots, d_{\ell} \ge 1$ be integers. 
    Suppose that $\mathcal{H}$ is an $n$-vertex $r$-graph and $\mathcal{E}_1, \mathcal{E}_2, \ldots, \mathcal{E}_{\ell} \subset \binom{V(\mathcal{H})}{i}$ are pairwise disjoint sets. 
    Then there exists a subgraph $\mathcal{H}' \subset \mathcal{H}$ such that $\mathcal{E}_j$ is $d_j$-full in $\mathcal{H}'$ for each $j\in [\ell]$ and 
    \begin{align*}
        |\mathcal{H}'|
        \ge |\mathcal{H}| - \sum_{j\in [\ell]}(d_j-1)|\mathcal{E}_{j}|. 
    \end{align*}
\end{lemma}
\begin{proof}[Proof of Lemma~\ref{LEMMA:full-subgraph}]
    We say a sequence $e_1, e_2, \ldots, e_m \in \bigcup_{j=1}^{\ell}\mathcal{E}_j$ is \textbf{$(d_1, \ldots,d_{\ell})$-sparse} if it satisfies the following conditions$\colon$
        \begin{itemize}
            \item $d_{\mathcal{H}}(e_1) \le d_j-1$, where $j\in [\ell]$ is the index such that $e_1 \in \mathcal{E}_j$. 
            \item For each $k>1$, the element $e_k$ is contained in fewer than $d_j$ edges of $\mathcal{H}$ that do not contain any of $e_1, \ldots, e_{k-1}$, where $j\in [\ell]$ is the index such that $e_k \in \mathcal{E}_j$. 
        \end{itemize}
    Fix a maximal $(d_1, \ldots,d_{\ell})$-sparse sequence $e_1, e_2, \ldots, e_m \in \bigcup_{j=1}^{\ell}\mathcal{E}_j$, and let $\mathcal{H}'$ be the $r$-graph obtained from $\mathcal{H}$ by deleting all edges that contain at least one element from $\{e_1, e_2, \ldots, e_m\}$. 
    It is clear from the definition and the maximality of $m$ that 
    \begin{align*}
        |\mathcal{H}'|
        \ge |\mathcal{H}| - \sum_{j\in [\ell]}(d_j-1)|\mathcal{E}_{j}|
    \end{align*}
    and each $\mathcal{E}_j$ is $d_j$-full in $\mathcal{H}'$.  
\end{proof}
Now we prove Theorem~\ref{THM:expansion-Zaran}~\ref{THM:expansion-Zaran-1}.
\begin{proof}[Proof of Theorem~\ref{THM:expansion-Zaran}~\ref{THM:expansion-Zaran-1}]
    Fix positive integers $m,n$ and let 
    \begin{align*}
        f(m,n)
        & \coloneqq (s_1+t_1)^2 (s_2+t_2) m^{2-\frac{1}{s_2}} n^{1-\frac{2}{s_1}}+ 2 t_1mn +2 s_1 n^{1+\frac{1}{s_1}}, \\
        r(m,n)
        & \coloneqq (s_1+1)(t_1+1)mn + (s_2+1)(t_2+1)m^2, \\
        g(m,n) 
        & \coloneqq t_1 mn^{1-\frac{1}{s_1}} + s_1 n, 
        \quad\text{and}\quad
        h(m)
        \coloneqq \frac{1}{2}(s_2 + t_2) m^{2-\frac{1}{s_2}}.
    \end{align*}
    It suffices to show that 
    \begin{align*}
        Z(m,n,K_{s_1,t_1}^{+},K_{s_2,t_2}^{+})
        \le 2\cdot f(m,n) + r(m,n).
    \end{align*}
    Suppose to the contrary that there exists an $m$ by $n$ semibiparite $3$-graph $\mathcal{H} = \mathcal{H}[V_1, V_2]$ such that 
    \begin{itemize}
        \item $|\mathcal{H}| > 2\cdot f(m,n) + r(m,n)$, 
        \item $\mathcal{H}[V_1, V_2]$ does not contain any ordered copy of $K_{s_1,t_1}^{+}$, and 
        \item $\mathcal{H}[V_1, V_2]$ does not contain any copy of $K_{s_2,t_2}^{+}$ whose core is contained in $V_1$. 
    \end{itemize}
    Let $G_1$ denote the induced bipartite subgraph of $\partial\mathcal{H}$ on $V_1$ and $V_2$, and let $G_2$ denote the induced subgraph of $\partial\mathcal{H}$ on $V_1$. 
    Let $d_i \coloneqq (s_i+1)(t_i+1)$ for $i\in \{1,2\}$. 
    By Lemma~\ref{LEMMA:full-subgraph}, there exists a  subgraph $\mathcal{H}' \subset \mathcal{H}$ such that $G_i$ is $d_i$-full in $\mathcal{H}'$ for $i\in \{1,2\}$, and 
    \begin{align}\label{equ:THM:expansion-Zaran-H'-lower}
        |\mathcal{H}'|
        & \ge |\mathcal{H}| - (d_1-1)|G_1| - (d_2-1)|G_2| \notag \\
        & \ge |\mathcal{H}| - (s_1+1)(t_1+1)mn - (s_2+1)(t_2+1)m^2
        > 2\cdot f(m,n). 
    \end{align}
    Let $G'_1$ denote the induced bipartite subgraph of $\partial\mathcal{H}'$ on $V_1$ and $V_2$, and let $G'_2$ denote the induced subgraph of $\partial\mathcal{H}'$ on $V_1$. 
    The following claim follows easily from the definition of $\mathcal{H}'$ and Fact~\ref{FACT:expansion-free-shadow}. 
    \begin{claim}\label{CLAIM:expansion-Zaran-G'1-G'2}
        The bipartite graph $G'_1[V_1, V_2]$ is ordered-$K_{s_1,t_1}$-free. 
        The graph $G'_2$ is $K_{s_2, t_2}$-free. 
    \end{claim}
    %
    It follows from Claim~\ref{CLAIM:expansion-Zaran-G'1-G'2} and Theorem~\ref{THM:KST54} that 
    \begin{align*}
        |G'_1|
         \le Z(m,n,K_{s_1,t_1})
        \le g(m,n) 
        \quad\text{and}\quad 
        |G'_2|
         \le \mathrm{ex}(m,K_{s_2,t_2})
        \le h(m). 
    \end{align*}
    Let $d'_1 \coloneqq \frac{f(m,n)}{2 \cdot g(m,n)}$ and $d'_2 \coloneqq \frac{f(m,n)}{2 \cdot h(m)}$. 
    It follows from Lemma~\ref{LEMMA:full-subgraph} and~\eqref{equ:THM:expansion-Zaran-H''-lower} that there exists a subgraph $\mathcal{H}'' \subset \mathcal{H}'$ such that $G'_i$ is $d'_i$-full in $\mathcal{H}''$ for each $i\in \{1,2\}$, and 
    \begin{align}\label{equ:THM:expansion-Zaran-H''-lower}
        |\mathcal{H}''|
        & \ge |\mathcal{H}'| - (d'_1-1)|G'_1| - (d'_2-1)|G'_2| \notag \\
        & > 2\cdot f(m,n) - \frac{f(m,n)}{2 \cdot g(m,n)} \cdot g(m,n) - \frac{f(m,n)}{2 \cdot h(m)} \cdot h(m)
        \ge f(m,n). 
    \end{align}
    Let $U_i \subset V_i$ be the collection of vertices whose degree is not zero in $\mathcal{H}''$ for $i\in \{1,2\}$. 
    Let $\tilde{m} \coloneqq |U_1|$ and $\tilde{n} \coloneqq |U_2|$. 
    Let $G''_1$ denote the induced bipartite subgraph of $\partial\mathcal{H}''$ on $U_1$ and $U_2$, and let $G''_2$ denote the induced subgraph of $\partial\mathcal{H}''$ on $U_1$.
    \begin{claim}\label{CLAIM:expansion-Zaran-G''-min-deg}
       The following statements hold. 
       \begin{enumerate}[label=(\roman*)]
           \item\label{CLAIM:expansion-Zaran-G''-min-deg-1} $d_{G''_1}(x) \ge d'_2$ and $d_{G''_2}(x) \ge  d'_1$ for every vertex $x\in U_1$.
           \item\label{CLAIM:expansion-Zaran-G''-min-deg-2} $d_{G''_1}(\tilde{x}) \ge d'_1$ for every vertex $\tilde{x}\in U_2$. 
       \end{enumerate}
       In particular, $\tilde{m} \ge d'_1 \ge \frac{f(m,n)}{2\cdot g(m,n)} > \frac{2 t_1mn +2 s_1 n^{1+\frac{1}{s_1}}}{2\cdot (t_1 mn^{1-\frac{1}{s_1}} + s_1 n)} = n^{\frac{1}{s_1}}$ and $\tilde{n} \ge d'_2 \ge \frac{f(m,n)}{2\cdot h(m)}$. 
    \end{claim}
    \begin{proof}[Proof of Claim~\ref{CLAIM:expansion-Zaran-G''-min-deg}]
        First, we prove Claim~\ref{CLAIM:expansion-Zaran-G''-min-deg}~\ref{CLAIM:expansion-Zaran-G''-min-deg-1}. 
        Fix a vertex $x\in U_1$. 
        It follows from the definition of $U_1$ that there exist vertices $y \in U_1$ and $z\in U_2$ such that $\{x,y,z\}\in \mathcal{H}''$.
        Note that $G''_i \subset G'_i$ is $d_i'$-full in $\mathcal{H}''$ for $i\in \{1,2\}$. 
        Therefore, the edge $xy \in G''_2$ satisfies $|N_{\mathcal{H}''}(xy)| \ge d_{\mathcal{H}''}(xy) \ge d'_2$, 
        and the edge $xz \in G''_1$ satisfies $|N_{\mathcal{H}''}(xz)| \ge d_{\mathcal{H}''}(xz) \ge d'_1$.
        Since $N_{\mathcal{H}''}(xy) \subset N_{G''_1}(x)$ and $N_{\mathcal{H}''}(xz) \subset N_{G''_2}(x)$, we obtain $d_{G''_1}(x) \ge |N_{\mathcal{H}''}(xy)| \ge d'_2$ and $d_{G''_2}(x) \ge |N_{\mathcal{H}''}(xz)| \ge d'_1$. 

        Next, we prove Claim~\ref{CLAIM:expansion-Zaran-G''-min-deg}~\ref{CLAIM:expansion-Zaran-G''-min-deg-2}. 
        Fix a vertex $\tilde{x}\in U_2$. 
        It follows from the definition of $U_2$ that there exist vertices $\tilde{y}, \tilde{z} \in U_1$ such that $\{\tilde{x}, \tilde{y}, \tilde{z}\}\in \mathcal{H}''$.
        Similar to the argument above, we have $d_{G''_1}(\tilde{x}) \ge d_{\mathcal{H}''}(\tilde{x}\tilde{y}) \ge d'_1$. 
    \end{proof}
    Recall that $G''_1[U_1, U_2]$ is ordered-$K_{s_1, t_1}$-free, so it follows from Theorem~\ref{THM:KST54} that
    \begin{align*}
        |G''_1|
        \le Z(\tilde{m}, \tilde{n}, K_{s_1,t_1})
        \le t_1 \tilde{m}\tilde{n}^{1-\frac{1}{s_1}} + s_1 \tilde{n}.
    \end{align*}
    By averaging, there exists a vertex $u_{\ast} \in U_1$ such that 
    \begin{align}\label{equ:THM:expansion-Zaran-u-ast-upper}
        d_{G''_1}(u_{\ast})
        \le \frac{|G''_1|}{\tilde{m}}
        \le \frac{t_1 \tilde{m}\tilde{n}^{1-\frac{1}{s_1}} + s_1 \tilde{n}}{\tilde{m}}
        \le t_1 \tilde{n}^{1-\frac{1}{s_1}} + s_1 \tilde{n}^{1-\frac{1}{s_1}}, 
    \end{align}
    where the last inequality follows from $\tilde{m} \ge n^{\frac{1}{s_1}} \ge \tilde{n}^{\frac{1}{s_1}}$ (see Claim~\ref{CLAIM:expansion-Zaran-G''-min-deg}). 

    Let $N_i \coloneqq N_{G''_i}(u_{\ast})$ for $i\in \{1,2\}$. Note that $N_1 \subset U_2$ and $N_2 \subset U_1$.
    \begin{claim}\label{CLAIM:expansion-Zaran-N1-N2}
        We have 
        \begin{align*}
            (s_1 + t_1) \tilde{n}^{1-\frac{1}{s_1}}
            \ge |N_1|
            \ge \frac{f(m,n)}{2\cdot h(m)}
            \quad\text{and}\quad 
            |N_2|
            \ge \frac{f(m,n)}{2\cdot g(m,n)} \ge n^{\frac{1}{s_1}}. 
        \end{align*}
    \end{claim}
    \begin{proof}[Proof of Claim~\ref{CLAIM:expansion-Zaran-N1-N2}]
        It follows from~\eqref{equ:THM:expansion-Zaran-u-ast-upper} that $|N_1| = d_{G''_1}(u_{\ast}) \le (s_1 + t_1) \tilde{n}^{1-\frac{1}{s_1}}$. 
        On the other hand, it follows from Claim~\ref{CLAIM:expansion-Zaran-G''-min-deg}~\ref{CLAIM:expansion-Zaran-G''-min-deg-1} that $|N_1| \ge d'_2 \ge \frac{f(m,n)}{2\cdot h(m)}$. 
        Similarly, it follows from Claim~\ref{CLAIM:expansion-Zaran-G''-min-deg}~\ref{CLAIM:expansion-Zaran-G''-min-deg-1} that $|N_2| \ge d'_1 \ge \frac{f(m,n)}{2\cdot g(m,n)} \ge n^{\frac{1}{s_1}}$. 
    \end{proof}
    Let $\widetilde{G} = \widetilde{G}[N_1, N_2]$ denote the induced bipartite subgraph of $\partial\mathcal{H}''$ on $N_1$ and $N_2$. 
    Similar to the proof of Claim~\ref{CLAIM:expansion-Zaran-G''-min-deg}~\ref{CLAIM:expansion-Zaran-G''-min-deg-1}, each vertex $x \in N_2$ has at least $|N_{\mathcal{H}''}(u_{\ast} x)| \ge d'_1 \ge \frac{f(m,n)}{2\cdot h(m)}$ neighbors (in $|\widetilde{G}|$) contained $N_{1}$. 
    Therefore, 
    \begin{align}\label{equ:THM:expansion-Zaran-tilde-G-lower}
        |\widetilde{G}|
        \ge |N_1| \cdot \frac{f(m,n)}{2\cdot h(m)}. 
    \end{align}
    On the other hand, notice that $\widetilde{G}[N_1, N_2]$ is ordered-$K_{s_1-1,t_1}$-free (since any ordered copy of $K_{s_1-1,t_1}$ in $\widetilde{G}$ would form an ordered copy of $K_{s_1,t_1}$ in $G''_1$). 
    So by Theorem~\ref{THM:KST54}, 
    \begin{align}\label{equ:THM:expansion-Zaran-tilde-G-upper}
        |\widetilde{G}|
        \le Z(|N_1|, |N_2|, K_{s_1-1,t_1}) 
        \le t_1 |N_2| |N_1|^{1-\frac{1}{s_1-1}} +s_1|N_1|.
    \end{align}
    Combining~\eqref{equ:THM:expansion-Zaran-tilde-G-lower} and~\eqref{equ:THM:expansion-Zaran-tilde-G-upper}, we obtain 
    \begin{align*}
        |N_2| \cdot \frac{f(m,n)}{2\cdot h(m)}
        \le t_1 |N_2| |N_1|^{1-\frac{1}{s_1-1}} +s_1|N_1|,
    \end{align*}
    which is equivalent to 
    \begin{align}\label{equ:THM:expansion-Zaran-fmn-over-hm}
        \frac{f(m,n)}{2\cdot h(m)}
        \le t_1 |N_1|^{1-\frac{1}{s_1-1}} +s_1\frac{|N_1|}{|N_2|}. 
    \end{align}
    It follows from Claim~\ref{CLAIM:expansion-Zaran-N1-N2} that 
    \begin{align*}
        t_1 |N_1|^{1-\frac{1}{s_1-1}}
        & \le t_1\left((s_1+t_1)\tilde{n}^{1-\frac{1}{s_1}}\right)^{1-\frac{1}{s_1-1}}
        < t_1 (s_1+t_1) n^{1-\frac{2}{s_1}}, \quad\text{and}\quad \\
        s_1\frac{|N_1|}{|N_2|}
        & \le \frac{s_1(s_1+t_1)\tilde{n}^{1-\frac{1}{s_1}}}{n^{\frac{1}{s_1}}} 
        \le s_1(s_1+t_1) n^{1-\frac{2}{s_1}}. 
    \end{align*}
    Combining with~\eqref{equ:THM:expansion-Zaran-fmn-over-hm}, we obtain 
    \begin{align*}
        f(m,n)
        & \le 2\cdot h(m) \cdot \left(t_1 (s_1+t_1) n^{1-\frac{2}{s_1}} + s_1(s_1+t_1) n^{1-\frac{2}{s_1}}\right) \\
        & = (s_1+t_1)^2(s_2+t_2) m^{2-\frac{1}{s_2}} n^{1-\frac{2}{s_1}}, 
    \end{align*}
    contradicting the definition of $f(m,n)$.  
\end{proof}

\subsection{Lower bound}
We prove Theorem~\ref{THM:expansion-Zaran}~\ref{THM:expansion-Zaran-3} in this subsection. 

Given a prime power $q$, let $\mathbb{F}_{q}$ denote the finite filed of size $q$, and let $\mathbb{F}_{q}^{\ast}$ denote the multiplicative subgroup of $\mathbb{F}_{q}$. 
For integers $s \ge 2$, the $\mathbb{F}_{q}$-norm on $\mathbb{F}_{q^{s-1}}$ is the map $N \colon \mathbb{F}_{q^{s-1}} \to \mathbb{F}_{q}$ defined by 
\begin{align*}
    N(x) \coloneqq x \cdot x^{q} \cdots x^{q^{s-2}} 
    \quad\text{for}\quad x \in \mathbb{F}_{q^{s-1}}. 
\end{align*}
The classical \textbf{projective norm graph} $\mathrm{PG}(q,s)$ introduced by Alon--R\'{o}nyai--Szab\'{o}~\cite{ARS99} is the graph with vertex set $\mathbb{F}_{q^{s-1}} \times \mathbb{F}_{q}^{\ast}$ where two distinct vertices $(X,x)$ and $(Y,y)$ are adjacent iff $N(X+Y) = xy$. 
It was shown in~\cite{ARS99} that 
\begin{enumerate}[label=(\roman*)]
    \item\label{PG-1} $\mathrm{PG}(q,s)$ has $q^{s}-q^{s-1}$ vertices, 
    \item\label{PG-2} every vertex in $\mathrm{PG}(q,s)$ has degree either $q^{s-1}-1$ or $q^{s-1}-2$ (and the latter case can happen only if $\mathrm{char}(\mathbb{F}_{q}) \neq 2$), 
    \item\label{PG-3} $\mathrm{PG}(q,s)$ has $\left(\frac{1}{2}-o(1)\right)q^{2s-1}$ edges, and 
    \item\label{PG-4} $\mathrm{PG}(q,s)$ is $K_{s,(s-1)!+1}$-free.
\end{enumerate}
The graph $\mathrm{PG}(q,s)$ is a well-known example of an optimal pseudorandom graph (see e.g.~\cite{Sza03,KS06,LMM22} for related definitions). 
Intuitively, a typical pair of vertices in $\mathrm{PG}(q,s)$ has around $q^{s-2}$ common neighbors, which is the expected number in a random graph with the same edge density. 
This intuition is made rigorous in the following lemma of Kostochka--Mubayi--Verstra\"{e}te~\cite{KMV15}. 
\begin{lemma}[{\cite[Lemmas~5.2 and 5.3]{KMV15}}]\label{LEMMA:KMV-PG-common-neighbor}
    Let $s\ge 3$ be an integer and $q$ be a prime power. 
    For every $(X,Y,x) \in \mathbb{F}_{q^{s-1}} \times \mathbb{F}_{q^{s-1}} \times \mathbb{F}_{q}^{\ast}$ with $X \neq Y$, the number of $Z \in \mathbb{F}_{q^{s-1}}$ satisfying $N \left( \frac{X+Z}{Y+Z}\right) = x$ is at least $q^{s-2}$.
    Consequently, for each $(X,x) \in \mathbb{F}_{q^{s-1}} \times \mathbb{F}_{q}^{\ast}$, all but at most $q-1$ vertices in $\mathbb{F}_{q^{s-1}} \times \mathbb{F}_{q}^{\ast}$ have at least $(1-o(1))q^{s-2}$ common neighbors with $(X,x)$ in $\mathrm{PG}(q,s)$. 
\end{lemma}

Now we are ready to present the construction for the lower bounds in Theorem~\ref{THM:expansion-Zaran}. 
Note that the construction applies to both $Z(n,n,K_{s_1,t_1}^{+}, K_{s_2,t_2}^{+})$ and $Z(n,n,K_{t_1,s_1}^{+}, K_{s_2,t_2}^{+})$. 
\begin{proof}[Proof of Theorem~\ref{THM:expansion-Zaran}~\ref{THM:expansion-Zaran-3}] 
    Let $s_1, s_2 \ge 3$ be integers and $p$ be a sufficiently large prime. 
    Let $(q, \tilde{q}) \coloneqq \left(p^{s_2}, p^{s_1}\right)$. 
    Note that $q^{s_1} = \tilde{q}^{s_2}$ and $q^{s_1} - q^{s_1-1} = (1+o(1))\left(\tilde{q}^{s_2} - \tilde{q}^{s_2-1}\right)$. 
    Let $n \coloneqq \max\left\{q^{s_1} - q^{s_1-1}, \tilde{q}^{s_2} - \tilde{q}^{s_2-1}\right\}$. 
    Let $V_1$ and $V_2$ be two disjoint sets of size $n$. 

    First, let us define a graph $H$ on $V_1 \cup V_2 \colon$ 
    \begin{itemize}
        \item place a copy of $\mathrm{PG}(\tilde{q},s_2)$ on $V_1$,  
        \item fix two injective maps $\psi_1 \colon \mathbb{F}_{q^{s_1-1}} \times \mathbb{F}_{q}^{\ast} \to V_1$ and $\psi_2 \colon \mathbb{F}_{q^{s_1-1}} \times \mathbb{F}_{q}^{\ast} \to V_2$, 
        \item add the pair $\left\{\psi_{1}\left((X,x)\right), \psi_{2}\left((Y,y)\right) \right\}$ to the edge set of $H$ iff $(X,x) \neq (Y,y)$ and $\left\{(X,x), (Y,y)\right\} \in \mathrm{PG}(q, s_1)$.  
    \end{itemize}
    Since $\mathrm{PG}(\tilde{q},s_2)$ is $K_{s_2, (s_2-1)!+1}$-free, the induced subgraph $H[V_1]$ is $K_{s_2, (s_2-1)!+1}$-free. 
    Similarly, since $\mathrm{PG}(q,s_1)$ is $K_{s_1, (s_1-1)!+1}$-free, it is easy to see that the induced bipartite subgraph $H[V_1, V_2]$ is $K_{s_1, (s_1-1)!+1}$-free. 

    Now, define the $n$ by $n$ semibipartite $3$-graph $\mathcal{H}$ as follows $\colon$ 
    \begin{align*}
        \mathcal{H} 
        \coloneqq \left\{\{u,v,w\} \colon \text{$u,v \in V_1$, $w \in V_2$, and $\{u,v,w\}$ forms a triangle in $H$} \right\}. 
    \end{align*}
    It is clear from the properties of $H$ that 
    \begin{itemize}
        \item $\mathcal{H}$ does not contain any copy of $K_{s_2,(s_2-1)!+1}^{+}$ whose core is contained in $V_1$, and 
        \item $\mathcal{H}[V_1, V_2]$ does not contain any ordered copy of $K_{s_1,(s_1-1)!+1}^{+}$ or $K_{(s_1-1)!+1, s_1}^{+}$.  
    \end{itemize}
    So it suffices to show that $|\mathcal{H}| \ge (1/2 - o(1)) n^{3-\frac{1}{s_2}-\frac{2}{s_1}}$. 
    Let $W \coloneqq \psi_{1}\left(\mathbb{F}_{q^{s_1-1}} \times \mathbb{F}_{q}^{\ast}\right) \subset V_1$. Note that $|W| = (1-o(1))|V_1|$. 
    By Lemma~\ref{LEMMA:KMV-PG-common-neighbor}, the number of edges in $H[W]$ that have at least $(1-o(1))q^{s-2}$ common neighbors in $H$ is at least 
    \begin{align*} 
        & \quad \frac{1}{2} \cdot  |V_1| \cdot \left(\tilde{q}^{s_2-1}-2 - (\tilde{q}-1)\right) - |V_1\setminus W| \cdot \tilde{q}^{s_2-1} \\ 
        & =  \frac{1}{2} \cdot \left(\tilde{q}^{s_2} - \tilde{q}^{s_2-1}\right) \cdot \left(\tilde{q}^{s_2-1}-\tilde{q}-1 \right) - o\left(\tilde{q}^{s_2}\right) \cdot \tilde{q}^{s_2-1} 
        = \left(\frac{1}{2}-o(1)\right)\tilde{q}^{2s_2-1}. 
    \end{align*}
    Therefore, 
    \begin{align*}
        |\mathcal{H}|
        \ge \left(\frac{1}{2}-o(1)\right)\tilde{q}^{2s_2-1} \cdot \left(\frac{1}{2}-o(1)\right) q^{s_1-2}
        & = \left(\frac{1}{2}-o(1)\right)\tilde{q}^{2s_2-1}q^{s_1-2} \\
        & = \left(\frac{1}{2}-o(1)\right) n^{3-\frac{1}{s_2}-\frac{2}{s_1}},  
    \end{align*}
    completing the proof of Theorem~\ref{THM:expansion-Zaran}~\ref{THM:expansion-Zaran-3}.    
\end{proof}
\section{Concluding remarks}\label{SEC:Remark}
Theorem~\ref{THM:weak-vtx-imply-boundedness} motivates the following question, which, if true, would imply that $K_{s,s}$ is bounded. 
\begin{problem}
    Let $s \ge 4$ be an integer. Is it true that 
    \begin{align*}
        \lim_{n\to \infty} \frac{Z(n,n,K_{s-1,s})}{\mathrm{ex}(n,K_{s,s})} = 0?
    \end{align*}
\end{problem}
Another interesting class of bipartite graphs studied by Erd\H{o}s--Simonovits~\cite{ES70} is the hypercube, where, 
for an integer $d\ge 2$, the $d$-dimensional hypercube $Q_d$ is the graph with vertex set $\{0,1\}^{d}$ and two vertices are adjacent iff they differ in exactly one coordinate.
\begin{problem}\label{PROB:boundedness-hypercube}
    Is $Q_d$ bounded for $ d\ge 3$? 
\end{problem}
Our understanding of the boundedness of degenerate $r$-graphs when $r \ge 3$ is very limited. 
We hope the following question will motivate further research on this topic. 
\begin{problem}\label{PROB:Characterize-bounded-r-graph}
    Characterize the family of bounded degenerate $r$-graphs.
\end{problem}
A particularly interesting class of hypergraph is $K_{s_1, \ldots, s_r}^r$, the complete $r$-partite $r$-graph with part sizes $s_1, \ldots, s_r$. 
It Tur\'{a}n number has been studied in works such as~\cite{Erdos64,Mu02,MYZ18,PZ21}. 
\begin{problem}\label{PROB:boundedness-complete-r-graph}
    Is $K_{s_1, \ldots, s_r}^{r}$ bounded for integers $r\ge 3$ and $s_r \ge \cdots \ge s_1 \ge 2$? 
\end{problem}
Our approach for Theorem~\ref{THM:Kst-expasion-bounded} could potentially be extended to the expansion of other bipartite graphs. 
This, in particular, motivates the following Zarankiewicz-type problem for $3$-graphs. 

Given a bipartite graph $F = F[U_1, U_2]$ with a proper bipartition $V(F) = U_1 \cup U_2$, we say a semibipartite $3$-graph $\mathcal{H}[V_1, V_2]$ contains an ordered copy of $F^{+}$ if there is a copy of $F^{+}$ in $\mathcal{H}$ such that $U_1$ is contained in $V_1$ and $U_2$ is contained in $V_2$. 
\begin{problem}\label{PROB:Zaran-3gp}
    Let $F[U_1, U_2]$ and $\widetilde{F}$ be two bipartite graphs. Determine the maximum number of edges $Z(m,n,F[U_1, U_2], \widetilde{F})$ in an $m$ by $n$ semibipartite $3$-graph $\mathcal{H}[V_1,V_2]$ such that 
    \begin{itemize}
        \item $\mathcal{H}[V_1, V_2]$ does not contain any ordered copy of $F^{+}$, and 
        \item $\mathcal{H}[V_1, V_2]$ does not contain any  copy of $\widetilde{F}^{+}$ whose core is contained in $V_1$.  
    \end{itemize}
\end{problem}

\section*{Acknowledgements}
XL would like to thank Jie Ma and Tianchi Yang for discussions on bipartite Tur\'{a}n problems, and Jie Ma for informing us about the definition of weak vertices and references~\cite{Sim84,Ma17}. 
\bibliographystyle{alpha}
\bibliography{boundedness}

\appendix

\section{Proof of Theorem~\ref{THM:expansion-Zaran}~\ref{THM:expansion-Zaran-2}}\label{SEC:Appendix}
\begin{proof}[Proof of Theorem~\ref{THM:expansion-Zaran}~\ref{THM:expansion-Zaran-2}]
    Fix positive integers $m,n$ and let 
    \begin{align*}
        f(m,n)
        & \coloneqq 2 (s_1+t_1)(s_2+t_2) \left(t_1 m^{2-\frac{1}{s_1}-\frac{2}{s_2}+\frac{1}{s_1s_2}} n+ s_1 m^{2+\frac{1}{s_1}-\frac{1}{s_2}} \right), \\
        r(m,n)
        & \coloneqq (s_1+1)(t_1+1)mn + (s_2+1)(t_2+1)m^2, \\
        g(m,n) 
        & \coloneqq t_1 n m^{1-\frac{1}{s_1}} + s_1 m, 
        \quad\text{and}\quad
        h(m)
        \coloneqq \frac{1}{2}(s_2 + t_2) m^{2-\frac{1}{s_2}}.
    \end{align*}
    It suffices to show that 
    \begin{align*}
        Z(m,n,K_{t_1,s_1}^{+},K_{s_2,t_2}^{+})
        \le 2\cdot f(m,n) + r(m,n).
    \end{align*}
    Suppose to the contrary that there exists an $m$ by $n$ semibiparite $3$-graph $\mathcal{H} = \mathcal{H}[V_1, V_2]$ such that 
    \begin{itemize}
        \item $|\mathcal{H}| > 2\cdot f(m,n) + r(m,n)$, 
        \item $\mathcal{H}[V_1, V_2]$ does not contain any ordered copy of $K_{t_1,s_1}^{+}$, and 
        \item $\mathcal{H}[V_1, V_2]$ does not contain any copy of $K_{s_2,t_2}^{+}$ whose core is contained in $V_1$. 
    \end{itemize}
    Let $G_1$ denote the induced bipartite subgraph of $\partial\mathcal{H}$ on $V_1$ and $V_2$, and let $G_2$ denote the induced subgraph of $\partial\mathcal{H}$ on $V_1$. 
    Let $d_i \coloneqq (s_i+1)(t_i+1)$ for $i\in \{1,2\}$. 
    By Lemma~\ref{LEMMA:full-subgraph}, there exists a  subgraph $\mathcal{H}' \subset \mathcal{H}$ such that $G_i$ is $d_i$-full in $\mathcal{H}'$ for $i\in \{1,2\}$, and 
    \begin{align}\label{equ:THM:expansion-Zaran-H'-lower-ii}
        |\mathcal{H}'|
        & \ge |\mathcal{H}| - (d_1-1)|G_1| - (d_2-1)|G_2| \notag \\
        & \ge |\mathcal{H}| - (s_1+1)(t_1+1)mn - (s_2+1)(t_2+1)m^2
        > 2\cdot f(m,n). 
    \end{align}
    Let $G'_1$ denote the induced bipartite subgraph of $\partial\mathcal{H}'$ on $V_1$ and $V_2$, and let $G'_2$ denote the induced subgraph of $\partial\mathcal{H}'$ on $V_1$. 
    The following claim follows easily from the definition of $\mathcal{H}'$ and Fact~\ref{FACT:expansion-free-shadow}. 
    \begin{claim}\label{CLAIM:expansion-Zaran-G'1-G'2-ii}
        The bipartite graph $G'_1[V_1, V_2]$ is ordered-$K_{t_1,s_1}$-free. 
        The graph $G'_2$ is $K_{s_2, t_2}$-free. 
    \end{claim}
    It follows from Claim~\ref{CLAIM:expansion-Zaran-G'1-G'2-ii} and Theorem~\ref{THM:KST54} that 
    \begin{align*}
        |G'_1|
         \le Z(m,n,K_{t_1,s_1})
        \le g(m,n) 
        \quad\text{and}\quad 
        |G'_2|
         \le \mathrm{ex}(m,K_{s_2,t_2})
        \le h(m). 
    \end{align*}
    Let $d'_1 \coloneqq \frac{f(m,n)}{2 \cdot g(m,n)}$ and $d'_2 \coloneqq \frac{f(m,n)}{2 \cdot h(m)}$. 
    It follows from Lemma~\ref{LEMMA:full-subgraph} and~\eqref{equ:THM:expansion-Zaran-H''-lower-ii} that there exists a subgraph $\mathcal{H}'' \subset \mathcal{H}'$ such that $G'_i$ is $d'_i$-full in $\mathcal{H}''$ for each $i\in \{1,2\}$, and 
    \begin{align}\label{equ:THM:expansion-Zaran-H''-lower-ii}
        |\mathcal{H}''|
        & \ge |\mathcal{H}'| - (d'_1-1)|G'_1| - (d'_2-1)|G'_2| \notag \\
        & > 2\cdot f(m,n) - \frac{f(m,n)}{2 \cdot g(m,n)} \cdot g(m,n) - \frac{f(m,n)}{2 \cdot h(m)} \cdot h(m)
        \ge f(m,n). 
    \end{align}
    Let $U_i \subset V_i$ be the collection of vertices whose degree is not zero in $\mathcal{H}''$ for $i\in \{1,2\}$. 
    Let $\tilde{m} \coloneqq |U_1|$ and $\tilde{n} \coloneqq |U_2|$. 
    Let $G''_1$ denote the induced bipartite subgraph of $\partial\mathcal{H}''$ on $U_1$ and $U_2$, and let $G''_2$ denote the induced subgraph of $\partial\mathcal{H}''$ on $U_1$.
    \begin{claim}\label{CLAIM:expansion-Zaran-G''-min-deg-ii}
       The following statements hold. 
       \begin{enumerate}[label=(\roman*)]
           \item\label{CLAIM:expansion-Zaran-G''-min-deg-ii-1} $d_{G''_1}(x) \ge d'_2$ and $d_{G''_2}(x) \ge  d'_1$ for every vertex $x\in U_1$.
           \item\label{CLAIM:expansion-Zaran-G''-min-deg-ii-2} $d_{G''_1}(\tilde{x}) \ge d'_1$ for every vertex $\tilde{x}\in U_2$. 
       \end{enumerate}
       In particular, $\tilde{m} \ge d'_1 \ge \frac{f(m,n)}{2\cdot g(m,n)}$ and $\tilde{n} \ge d'_2 \ge \frac{f(m,n)}{2\cdot h(m)} > \frac{2(s_1+t_1)(s_2+t_2)s_1 m^{2+\frac{1}{s_1}-\frac{1}{s_2}}}{(s_2+t_2)m^{2-\frac{1}{s_2}}/2} > m^{\frac{1}{s_1}}$. 
    \end{claim}
    \begin{proof}[Proof of Claim~\ref{CLAIM:expansion-Zaran-G''-min-deg-ii}]
        First, we prove Claim~\ref{CLAIM:expansion-Zaran-G''-min-deg-ii}~\ref{CLAIM:expansion-Zaran-G''-min-deg-ii-1}. 
        Fix a vertex $x\in U_1$. 
        It follows from the definition of $U_1$ that there exist vertices $y \in U_1$ and $z\in U_2$ such that $\{x,y,z\}\in \mathcal{H}''$.
        Note that $G''_i \subset G'_i$ is $d_i'$-full in $\mathcal{H}''$ for $i\in \{1,2\}$. 
        Therefore, the edge $xy \in G''_2$ satisfies $|N_{\mathcal{H}''}(xy)| \ge d_{\mathcal{H}''}(xy) \ge d'_2$, 
        and the edge $xz \in G''_1$ satisfies $|N_{\mathcal{H}''}(xz)| \ge d_{\mathcal{H}''}(xz) \ge d'_1$.
        Since $N_{\mathcal{H}''}(xy) \subset N_{G''_1}(x)$ and $N_{\mathcal{H}''}(xz) \subset N_{G''_2}(x)$, we obtain $d_{G''_1}(x) \ge |N_{\mathcal{H}''}(xy)| \ge d'_2$ and $d_{G''_2}(x) \ge |N_{\mathcal{H}''}(xz)| \ge d'_1$. 

        Next, we prove Claim~\ref{CLAIM:expansion-Zaran-G''-min-deg-ii}~\ref{CLAIM:expansion-Zaran-G''-min-deg-ii-2}. 
        Fix a vertex $\tilde{x}\in U_2$. 
        It follows from the definition of $U_2$ that there exist vertices $\tilde{y}, \tilde{z} \in U_1$ such that $\{\tilde{x}, \tilde{y}, \tilde{z}\}\in \mathcal{H}''$.
        Similar to the argument above, we have $d_{G''_1}(\tilde{x}) \ge d_{\mathcal{H}''}(\tilde{x}\tilde{y}) \ge d'_1$. 
    \end{proof}
    Recall that $G''_1[U_1, U_2]$ is ordered-$K_{s_1, t_1}$-free, so it follows from Theorem~\ref{THM:KST54} that
    \begin{align*}
        |G''_1|
        \le Z(\tilde{m}, \tilde{n}, K_{t_1,s_1})
        \le t_1 \tilde{n} \tilde{m}^{1-\frac{1}{s_1}} + s_1 \tilde{m}.
    \end{align*}
    By averaging, there exists a vertex $u_{\ast} \in U_2$ such that 
    \begin{align}\label{equ:THM:expansion-Zaran-u-ast-upper-ii}
        d_{G''_1}(u_{\ast})
        \le \frac{|G''_1|}{\tilde{n}}
        \le \frac{t_1 \tilde{n} \tilde{m}^{1-\frac{1}{s_1}} + s_1 \tilde{m}}{\tilde{n}}
        \le t_1 \tilde{m}^{1-\frac{1}{s_1}} + s_1 \tilde{m}^{1-\frac{1}{s_1}}, 
    \end{align}
    where the last inequality follows from $\tilde{n} \ge m^{\frac{1}{s_1}} \ge \tilde{m}^{\frac{1}{s_1}}$ (see Claim~\ref{CLAIM:expansion-Zaran-G''-min-deg-ii}). 

    Let $N \coloneqq N_{G''_1}(u_{\ast}) \subset U_1$. Note from~\eqref{equ:THM:expansion-Zaran-u-ast-upper-ii} that $|N| \le (s_1+t_1) \tilde{m}^{1-\frac{1}{s_1}}$. 
    Combined with the $K_{s_2, t_2}$-freeness of $G''_2[N]$ and Theorem~\ref{THM:KST54}, we obtain 
    \begin{align}\label{equ:expansion-Zaran-N-upper}
        |G''_2[N]|
        \le \frac{1}{2}(s_2+t_2)|N|^{2-\frac{1}{s_2}}. 
    \end{align}
    On the other hand, similar to the proof Claim~\ref{CLAIM:expansion-Zaran-G''-min-deg-ii}~\ref{CLAIM:expansion-Zaran-G''-min-deg-ii-1}, each vertex $x \in N$ has at least $d_{\mathcal{H}''}(u_{\ast}x) \ge d'_1$ neighbors (in $G''_2$) contained in $N$. 
    Therefore,  
    \begin{align}\label{equ:expansion-Zaran-N-lower}
        |G''_2[N]|
        \ge  \frac{|N| d'_1}{2} 
        \ge |N| \frac{f(m,n)}{4 \cdot g(m,n)}
    \end{align}
    Combining~\eqref{equ:expansion-Zaran-N-upper} and~\eqref{equ:expansion-Zaran-N-lower}, we obtain 
    \begin{align*}
        |N| \frac{f(m,n)}{4 \cdot g(m,n)}
        \le \frac{1}{2}(s_2+t_2)|N|^{2-\frac{1}{s_2}}, 
    \end{align*}
    which implies that 
    \begin{align*}
        f(m,n)
        & \le 2(s_2+t_2) \cdot g(m,n) \cdot |N|^{1-\frac{1}{s_2}} \\
        & \le 2(s_2+t_2) \cdot \left(t_1 n m^{1-\frac{1}{s_1}} + s_1 m\right) \cdot \left((s_1+t_1) \tilde{m}^{1-\frac{1}{s_1}}\right)^{1-\frac{1}{s_2}} \\
        & < 2(s_2+t_2)(s_1+t_1) \left(t_1 n m^{(1-s_{1}^{-1})(2-s_2^{-1})} + s_1 m^{1+(1-s_{1}^{-1})(1-s_2^{-1})}\right), 
    \end{align*}
    contradicting the definition of $f(m,n)$. 
\end{proof}
\end{document}